\documentclass[12pt,reqno]{amsart}

\usepackage{mathrsfs}
\usepackage{multicol}

\usepackage{amssymb,latexsym,amsbsy,amsthm,amsxtra,amscd,amsfonts,amsthm,amsmath,verbatim}

\usepackage[all,cmtip]{xy}
\usepackage{lmodern}
\usepackage[T1]{fontenc}

\usepackage[margin=1in]{geometry}
\usepackage{paralist}

\input xy
\xyoption{all}

\usepackage[colorlinks,pagebackref=true]{hyperref}

\usepackage{mathrsfs}
\usepackage{multicol}
\usepackage{fancyvrb}
\usepackage{color}
\usepackage{amssymb,latexsym,amsbsy,amsthm,amsxtra,amscd,amsfonts,amsthm,amsmath,verbatim}

\makeatletter

\@addtoreset{equation}{section}
\def\theequation{\thesection.\@arabic \c@equation}

\def\@citecolor{blue}
\def\@linkcolor{blue}
\def\@urlcolor{blue}

\makeatother

\usepackage{amssymb,latexsym,color,
amsxtra,amscd,amsfonts,amsthm,amsmath,verbatim,epsfig}

\usepackage{multicol}
\usepackage{fancyvrb}
\usepackage[colorlinks,pagebackref=true]{hyperref}
\def\@citecolor{blue}
\def\@linkcolor{blue}
\def\@urlcolor{blue}

\def\theequation{\arabic{equation}}
\def\theequation{\thesection.\arabic{equation}}
\numberwithin{equation}{section}

\def\ann{\operatorname{Ann}}
\def\ass{\operatorname{Ass}}

\def\deg{\operatorname{deg}}

\def\depth{\operatorname{depth}}
\def\dim{\operatorname{dim}}

\def\height{\operatorname{ht}}

\def\ker{\operatorname{ker}}

\def\reg{\operatorname{reg}}

\newcommand{\kk}{\Bbbk}

\def\lrar{{\longrightarrow}}

\def\P{\mathbb P}
\def\F{\mathcal F}
\def\R{\mathcal R}
\def\I{\mathcal I}

\newcommand{\ncom}{\newcommand}
\ncom{\bq}{\begin{equation}}
\ncom{\eq}{\end{equation}}
\ncom{\beqn}{\begin{eqnarray*}}
\ncom{\eeqn}{\end{eqnarray*}}
\ncom{\beq}{\begin{eqnarray}}
\ncom{\eeq}{\end{eqnarray}}
\ncom{\been}{\begin{enumerate}}
\ncom{\eeen}{\end{enumerate}}
\ncom{\olin}{\overline}
\ncom{\f}{\frac}
\ncom{\rar}{\rightarrow}

\def\nno{\nonumber}

\newcommand{\p}{\mathfrak p}

\newcommand{\m}{\mathfrak m}

\theoremstyle{plain}
\newtheorem{theorem}[equation]{Theorem}
\newtheorem{corollary}[equation]{Corollary}
\newtheorem{proposition}[equation]{Proposition}
\newtheorem{lemma}[equation]{Lemma}

\theoremstyle{definition}

\def\sta{\stackrel}
\ncom{\bib}{\bibitem}

\ncom{\limnn}{\underset{\underset{}{n \longrightarrow \infty}}{\lim}}
\ncom{\limns}{\underset{\underset{}{s \longrightarrow \infty}}{\lim}}
\ncom{\maxi}{\underset{\underset{}{i}}{\max}}

\ncom{\limm}{\underset{{ n \to \infty}}{\lim}}
\ncom{\Tprime}{T^{\prime}}
\ncom{\mprime}{\m^{\prime}}

\def\z{{\bf z}}

\begin{document}
 \title[ ] {  Symbolic blowup algebras and invariants associated to certain monomial curves in $\P^3$  }
 \author{Clare D'Cruz$^*$}
 \address{Chennai Mathematical Institute, Plot H1 SIPCOT IT Park, Siruseri, 
Kelambakkam 603103, Tamil Nadu, 
India
} 
\email{clare@cmi.ac.in} 

\author[Mousumi Mandal] {Mousumi Mandal$^\dag$}
\address{Department of Mathematics, Indian Institute of Technology Kharagpur, 721302, India}
\email{mousumi@maths.iitkgp.ac.in}

\keywords{Symbolic  Rees algebra, resurgence, Waldschmidt  constant, Castelnuovo-Mumford regularity}
\thanks{$^*$ partially supported by a grant from Infosys Foundation}
 \thanks{$^\dag$ Supported by ISIRD(IIT KGP) grant No.: IIT/SRIC/MA/ANL/2017-18/45, India}
 \subjclass[2010]{Primary: 13A30, 1305, 13H15, 13P10} 

\begin{abstract}
 In this paper we explicitly describe the symbolic powers of the ideal defining the  curve ${\mathcal C}(q,m)$ in $\P^3$  parametrized by 
 $(  x^{d+2m}, x^{d+m} y^m, x^{d} y^{2m}, y^{d+2m})$, where  $q,m$ are positive integers,  $d=2q+1$ and   $\gcd(d,m)=1$. 
 We show that the symbolic blowup algebra is Noetherian and Gorenstein. An explicit  formula for the resurgence and the Waldschmidt constant of the prime ideal $\p:=\p_{ {  \mathcal C}(q,m) }$ defining the curve ${\mathcal C}(q,m)$ is computed.  We also give a formula for the Castelnuovo-Mumford regularity of the symbolic powers $\p^{(n)}$ for all $n \geq 1$. 
  \end{abstract}

 \maketitle
  \section{Introduction}
 
  Let $\kk$ be a field, $A= \kk[x_1, \ldots, x_t]$  a polynomial ring and $I$  a homogeneous ideal in $A$ with no embedded components. For every $n 
 \geq 1$, the $n$-th 
  symbolic power of $I$ is defined as $I^{(n)} := \displaystyle{\bigcap_{\p \in \ass(A/I) } (I^n A_{\p}\cap A)}$. By a classical result of Zariski and Nagata  $n$-th symbolic power of a given prime ideal consists of the elements that vanish up to order $n$ on the corresponding variety.
    However, describing the generators of symbolic powers is not easy.
One can verify  that $I^n\subseteq I^{(n)}$ and in fact for $0  \not= I\subset A,$ $I^r\subseteq I^{(n)}$ holds if and only if $r \geq n$.  
It is a challenging problem to determine for which $n$ and $r$ the containment  $I^{(n)}\subseteq I^r$ holds true. The
    results in \cite {ein-laz-smith} and \cite{ {hochster-huneke}} show that $I^{(n)}\subseteq I^r$ for $n\geq (t-1)r$. In the direction of comparing the symbolic powers and ordinary powers of ideals,  B.~Harbourne raised the following conjecture in \cite[Conjecture~8.4.3]{BDHKKASS}:   For any homogeneous ideal $I \subset A$, $I^{(n)} \subseteq I^r $ if $n \geq r(t-1) - (t-2)$. 
It is of interest to study the least integer $n$ for which $I^{(n)}\subseteq I^r$ holds for a given ideal $I$ and  for an integer $r$. To answer this question C. Bocci and B. Harbourne  defined an asymptotic quantity called resurgence which is defined as $\rho(I) = \sup\{ m/r \mid 
  I^{(m)}  \not \subseteq I^r\}$ (see \cite{BH}). From the results in \cite {ein-laz-smith} and \cite{ {hochster-huneke}} it follows that this quantity exists for radical ideals.
In fact,   $1\leq \rho(I)\leq t-1$ (see \cite{BH}).   Since it is hard to compute the exact value of resurgence, in the same paper \cite{BH},  they  defined another invariant 
   was first introduced by Waldschmidt  in \cite{waldschmidt}. They call this invariant the Waldschmidt constant and denote it by $\gamma(I)$ in \cite{BH}. 
This invariant is defined as
  ${\gamma(I) }= \displaystyle{\lim_{n\rightarrow \infty}}\frac{\alpha(I^{(n)}) }{n}$, where  $ \alpha(I)$ denotes  the least degree of a homogeneous generator of $I$. They showed that if $I$ is an homogeneous ideal, then $\alpha(I) / {\gamma}(I) \leq \rho(I)$ and in addition if $I$ is a  zero dimensional subscheme  in a projective space, then
   $\alpha(I) / {\gamma}(I) \leq \rho(I) \leq \reg (I) / {\gamma}(I)$, where $\reg(I)$ denotes the Castelnuovo-Mumford regularity \cite[Theorem~1.2.1]{BH}. Hence, if  $\alpha(I)=\reg(I)$, then $\rho(I)=\alpha(I)/{\gamma}(I)$.
   Later,  in \cite[Conjecture~2.1]{harb-huneke} Harbourne and Huneke raised the following Conjecture: Let $I$ be an ideal of fat points  in $A$ and $\m = (x_1, \ldots, x_t)$. Then   $I^{ (n(t-1) )  } \subseteq \m^{n(t-2)} I^n$ holds true for all $I$ and  $n$. In the same paper they  showed that the conjecture is true for fat point ideals arising as
  symbolic powers of radical ideals generated in a single degree in  $\P^2$.  

 The resurgence  and the Waldschmidt constant has been studied in a few cases: for certain general points in $\P^2$  \cite{BH0},  smooth subschemes  \cite{guardo}, fat linear subspaces 
 \cite{fatabbi},  special point configurations \cite{duminicki} and monomial ideals \cite{bocco-waldschmidt}. The behaviour of  Castelnuovo-Mumford regularity of symbolic powers  is not easy to predict. From a 
  result of Cutkosky, Herzog and Trung, it follows that if $I$ is an ideal of points in a projective space and the symbolic Rees algebra $\displaystyle{\bigoplus_{n \geq 0} } I^{(n)}$ is Noetherian, then  
  $\reg(I^{(n)})$ is a quasi-polynomial  (\cite[Theorem~4.3]{cutkosky1}). Moreover, ${\displaystyle \lim_{n\rightarrow \infty}\left( \f{\reg(I^{(n)})}{n} \right)}$ exists and can even be irrational  \cite{cutkosky2}.

Though there are several results available for the the resurgence, the Waldschmidt constant and the Castelnuovo-Mumford regularity of symbolic powers, there is no precise result for monomial curves in a projective space. Though it is well known that $\p^{(2)}/\p^2$ is a cyclic module (for example see \cite[Lemma~2.1]{morales-simis}), the explicit description of the generator has been crucial in our study of the various invariants. 
  In this paper we focus on  the ideal defining the monomial curve  ${\mathcal C}(q,m)$ in $\P^3$  parametrized by 
 $(  x^{d+2m}, x^{d+m} y^m, x^{d} y^{2m}, y^{d+2m})$, where  $q,m$ are positive integers, $d=2q+1$ with $\gcd(d,m)=1$. 
 
  Another  topic of interest is the Gorenstein property of the symbolic blowup algebras. If $q=1$, i.e., $d=3$, then the monomial curves we consider coincide with the curves in \cite[Theorem~3.2(v)(2)(a)]{morales-simis}.  However, our  proof here is different. 
We use   properties of   monomial ideals to obtain our results. These computations are also useful in computing the Castelnuovo-Mumford regularity of the symbolic powers. The Gorenstein property of monomial curves in $\P^3$ have also been studied by Schenzel in \cite{schenzel}. However, their  curves do not overlap with  the  monomial curves we consider in this paper.

 We now briefly describe the contents of  our paper. Let $\phi: R = \kk[x_1, x_2, x_3, x_4]\lrar S = \kk[x,y]$ be a homomorphism  given by
  $\phi(x_1) = x^{d+2m}$,  $\phi(x_2) = x^{d+m} y^{m}$,    $\phi(x_3) = x^{d} y^{2m}$    and $\phi(x_4) = y^{d+2m}$.    For the rest of this paper $\p := \p_{ {  \mathcal C}(q,m) } := \ker~\phi$. In Section~2, we prove a few preliminary results. 
  In Section~3, we prove some results on monomial ideals which will be used in the subsequent sections.  
In Section~4, we explicitly describe the generators of $\p^{(n)}$ (Theorem~\ref{symbolic power}) and show that the symbolic Rees algebra $\displaystyle{\bigoplus_{n \geq 0} \p^{(n)}}$ is Noetherian. As computing symbolic powers is not easy we use a simple trick. We consider the ring  $T = R/ (x_1, x_4)$.  Then  $\p T$ is a monomial ideal and eventually  for all
 $n \geq 1 $, $\p^{(n)}T$ is a monomial ideal (Proposition~\ref{main theorem 1}).  In Section~5, we show that $\R_s(\p)$ is Gorenstein. Moreover, the symbolic fiber cone ${\mathcal F}_s(\p) := \R(\p) \otimes_R R/\m = \displaystyle{\bigoplus_{n \geq 0} \p^{(n)}  /  \m \p^{(n)}}$ is Cohen-Macaulay (Theorem~\ref{cm-rees}). 
 
Section 6 is devoted to study certain invariants associated to   $\p$, namely the resurgence, the Waldschmidt constant and the Castelnuovo-Mumford regularity. 
    We verify that  Conjecture~8.4.3 in \cite{BDHKKASS}  and Conjecture~2.1 in \cite{harb-huneke} is true for 
   $\p$ (Corollary~\ref{corollary containment 1}, Corollary~\ref{corollary containment 1}). We express  the resurgence for the monomial curve  ${\mathcal C}(q,m)$  in terms of the degree of the curve ${  \mathcal C}(q,m)$. In particular we show that  
   ${\displaystyle \rho(\p) = \f{e(R/ \p) -1}{ e(R / \p) -2}}$,  where $e(R/ \p)$ is the degree of ${ \mathcal C}(q,m)$ (Theorem~\ref{theorem-resurgence}). 
The Waldschmidt constant is  calculated for the same (Theorem~\ref{theorem-waldschmidt}).  
 We next give an explicit formula for  the  Castelnuovo-Mumford regularity for all the symbolic powers  $\p^{(n)}$ and show that it is a quasi-polynomial (Theorem~\ref{theorem-regularity}). As a consequence we show that ${\displaystyle \limnn \reg\left( \f{R}{ \p^{(n) } } \right) = \f{e(R/ \p) }{2}}$ (Corollary~\ref{limit regularity}). 
We end this paper by  comparing all these invariants and show that there exist monomial curves for which Theorem 1.2.1(b)  in \cite{BH} may not  hold true (Lemma~\ref{negative result}).

 \section{Basic results}
  \label{prelim}
 In this section we prove several results which are probably well known in literature. We provide proofs for the sake of convenience. 
  For the rest of  this paper $q$, $m$ and $d$ are as in the introduction and $\p :=  \p_{{  \mathcal C}(q,m)} \subseteq R$. 
  
 It is well known that the generators for $\p$  are the $2 \times 2$ minors of the matrix 
$
{\displaystyle
\begin{pmatrix}
  x_1   &  x_2 &  x_3^{q+m}\\
   x_2   &  x_3  &  x_1^q x_4^m
\end{pmatrix}
}$ \cite{morales}. 
In particular, if  
\beq
\label{defining equations}
  g_1 =  x_1^q x_2  x_4^m - x_3^{q+m+1} ,    \hspace{.2in}
 g_2 = x_1^{q+1} x_4^{m} - x_2  x_3^{q+m}   \hspace{.2in}
 \mbox{ and } 
 g_3 =  x_1x_3 - x_2^2   
 \eeq  then 
  $\p = (g_1, g_2, g_3)$.

\begin{lemma}
\label{cohen-macaulay}
  \been
     \item 
      \label{cohen-macaulay-1}
  $R/ \p$ is Cohen-Macaulay. In particular $x_1, x_4$ is a  regular sequence in $R/ \p$.
  
  \item
   \label{corollary multiplicity}
$
    {\displaystyle e   \left( (x_1,x_4); \f{R}{\p}\right)= 2(q+m) + 1 = d+2m.}
   $
    \eeen
  \end{lemma}
  \begin{proof} 
  
(\ref{cohen-macaulay-1})  From the Hilbert-Burch theorem it follows that the minimal  free resolution of $\p$ is of the form
  \beq
  \label{minimal free resolution}
  0  
\lrar R[-(q+m+2)]^2  
\sta{\phi}{\lrar} R[-(q+m+1)]^2 \oplus R[-2]   
\sta{\psi} {\lrar}  \p 
 \lrar 0
  \eeq
 where 
 $
 {\displaystyle
\phi = \begin{pmatrix}
x_2 & x_1 \\
-x_3 & -x_2\\
x_1^qx_4^m & x_3^{q+m}
\end{pmatrix}
 }$ 
 and $
 {\displaystyle
\psi = \begin{pmatrix}
g_1 & g_2 & g_3\\
  \end{pmatrix}
 }$.
 Hence $\depth(R/ \p) =2=\dim(R/ \p) =2$. This implies that  $R/ \p$ is Cohen-Macaulay. As $x_1, x_4$ is a system of parameters for $R/\p$, it is a regular sequence by \cite[Corollary~11.12]{hio}. 
 
 (\ref{corollary multiplicity})
Since $x_1, x_4$ is a regular sequence, 
 \beqn
 e   \left( (x_1,x_4); \f{R}{\p}\right)
 = \ell \left(   \f{R} {(\p + ( x_1, x_4))}\right)
 = \ell \left( \f{R}{(x_1, x_4, x_2^2, x_3^{2(q+m)+1})}\right)
  =  2(q+m)+1 .
  \eeqn
\end{proof}

Put
  \beq
  \label{definition of f2}
 f 
  &:=& 
  x_3^{q+m} g_1 - x_1^{q} x_4^m g_2 + x_1^{q-1}  x_3^{q+m-1}x_2 x_4^m g_3\\
&=&   - x_3^{2(q+m)+1} 
  -  x_1^{q-1}  x_2^3 x_3^{q+m-1}  x_4^{m} 
  +  3 x_1^{q}   x_2  x_3^{q+m} x_4^{m}
   - x_1^{2q+1}  x_4^{2m}. 
   \eeq

\begin{lemma}
\label{lemma on f_2}
    \been
 \item \label{lemma on f_2-1}
 For  $i=1,2,3$, $x_i f \in \p^2$.
 \item    \label{lemma on f_2-2}
   $f \in \p^{(2)}$.
 \item \label{lemma on f_2-3}
 For all  $n = 1, \ldots, q+m+1$, $f^n \in \p^{2n-1}$.  
 \eeen
\end{lemma}
\begin{proof}
(\ref{lemma on f_2-1})
As $g_i \in \p$ for $i=1,2,3$, 
\beq \nno
\label{g_i and f_2}
x_1 f &=& x_3^{q+m-1} g_1 g_3 - g_2^2 \in \p^2\\ \nno
x_2 f &=& -x_1^{q-1}x_3^{q+m-1} x_4^m g_3^2 - g_1 g_2 \in \p^2\\
x_3f &=&- x_1^{q-1} x_4^m g_2 g_3 - g_1^2 \in \p^2.
\eeq

(\ref{lemma on f_2-2}) From (1) it follows that $x_1 f \in \p^2 \subseteq \p^{(2)}$. As $x_1  \not \in \p$,   $f \in \p^{(2)}$. 

(\ref{lemma on f_2-3})
Let  $1 \leq n  \leq  q+m+1$. By the definition of $f$,  
\beq
 \label{syz-f_2-1} \nno
             f^n
& =& (x_3^{q+m} g_1 - x_1^{q} x_4^m g_2 + x_1^{q-1}  x_3^{q+m-1}x_2 x_4^m g_3)f^{n-1}\\ \nno
   & =& (  x_3^{q+m}    f^{n-1})g_1 
   -  ( x_1^{q} x_4^m f^{n-1})  g_2  
   +   ( x_1^{q-1}  x_3^{q+m-1}x_2 x_4^m f^{n-1} )   g_3\\ \nno
   & \in & \p^{2(n-1)} \p     \hspace{3.2in} \mbox{[from (\ref{lemma on f_2-1})]}  \\
   &=&  \p^{2n-1}.
  \eeq
\end{proof}

\section{Computations with monomial ideals}
 In general, symbolic powers are not easy to compute. Hence, we first consider the ring  $T:= R/(x_1, x_4) \cong \kk[x_2, x_3]$.  Since $\p T$ is a monomial ideal, $\p^n T$ is also. Consider 
  \beq
\label{definition of Ji}
\p T =  ( x_2^2, x_2 x_3^{q+m}, x_3^{q+m+1}  ),   
\hspace{.2in}
(f) T  = (x_3^{2(m+q)+1}  ), 
\hspace{.2in} 
I_n  := 
  \sum_{n_1 + 2 n_2 = n} (f^{n_{2}}T ) (\p T)^{n_1}  \subseteq \p^{(n)} T.
\eeq
Our aim in this section is to compute  $\ell (T / I_n)$. For this, we fist need to show that  $(I_n :x_3^{q+m})  \subseteq I_{n-1}$. Next we will compute $\ell (I_{n-1} / (I_n :x_3^{q+m}))$. 

\begin{lemma}
\label{ideal containment 1}
For all $n \geq 2$, 
$ (I_n : x_3^{q+m})  \subseteq I_{n-1}$.
\end{lemma}
\begin{proof}
From the definition of $I_n$ we get
\beqn
                        (I_n : x_3^{q+m} )
&=&               \sum_{a_1+2a_2=n; a_2 \not = 0} (f^{a_2} (\p T  )^{a_1}  : x_3^{q+m}) 
+                    ( ( x_2^2, x_2 x_3^{q+m}, x_3^{q+m+1}  )^n:x_3^{q+m})\\
&=&                \sum_{a_1+2a_2=n;a_2 \not =0}(x_3^{q+m+1} ) f^{a_2-1} (\p T)^{ a_1} 
+                      (   x_2^{2n}) 
+                     \sum_{i=1}^n    ( x_3^{(q+m)(i-1)} (x_2, x_3)^{i-1} x^{2(n-i)}_2 ) (x_2, x_3)\\
&\subseteq &   I_{n-1}.
\eeqn
\end{proof}

Our next step is to describe the generating set of $I_n$ modulo $(I_n :x_3^{q+m})$. 

\begin{lemma} 
\label{vector basis}
The minimal set of generators of $I_{n-1} / (I_n :x_3^{q+m})$ form a vector space basis over $\kk$. 
\end{lemma}
\begin{proof}
Put $M$ = $I_{n-1} / (I_n :x_3^{q+m})$ and  $\m^{\prime} = (x_2, x_3)$. Since  $x_3^{q+m} \m^{\prime} I_{n-1} \subseteq (\p T) I_{n-1} \subseteq I_n$, we get
$   \m^{\prime} I_{n-1} \subseteq   (I_n : x_3^{q+m})$. Hence $\m^{\prime} M   =0$ which implies that $M /\m^{\prime} M \cong M$.  By graded Nakayama's Lemma the generators of $M$ form a vector space basis over $T/ \m^{\prime} \cong \kk$.  
\end{proof}

In Lemma~\ref{ideal containment} we explicitly describe the generating set of $I_{n-1}$ modulo  $(I_n :x_3^{q+m})$.  We state a result on monomial ideals which follows from \cite[Proposition~1.14]{ene-herzog} and will be consistently used in all the proofs which involve monomial ideals. 
\begin{proposition}
\label{prop-herzog}
Let $I = (u_1, \ldots, u_r)$ and $J = (v)$ be monomial ideals  in a polynomial ring over a field $\kk$. 
Then $I : J =  ( \{u_i/gcd(u_i,v) : i = 1,...,r\})$.
\end{proposition}

\begin{lemma} 
\label{in between step} For all $n \geq 1$, 
$\p^{n} T \subseteq x_2^{2n-1}(x_2, x_3^{q+m} ) +( I_{n + 1}: x_3^{q+m})$
\end{lemma}
\begin{proof}
We prove by induction on $n$.   If $n =1$, 
then 
$$                       \p T 
=                x_2( x_2, x_3^{q+m} ) + (x_3^{q+m+1}) \subseteq x_2( x_2, x_3^{q+m})  +( I_2: x_3^{q+m}).
$$
Hence the  claim  is true for $n=1$. 
Let $n>1$. Then
\beq
\label{in between step 1}
 \nno
                      && \p^{n} T  \\ \nno
&=&                 (\p T) ( \p^{n-1} T)\\ \nno
& \subseteq&   (( x_2^2,  x_2x_3^{q+m} ), x_3^{q+m+1} ) 
                       \left(  x_2^{2n-3}  (x_2, x_3^{q+m}) 
  +                         (I_{n} : x_3^{q+m} )\right)
                         \hspace{1.0in} \mbox{[by induction hypothesis]}     \\   \nno 
&=&              x_2^{2n-1}  (x_2, x_3^{q+m})     
+                         (x_2^{2n-2} x_3^{2(q+m)} ) 
+                        (x_2^2, x_2 x_3^{q+m})       (I_{n} : x_3^{q+m} )    \\  
&&+                  (fT) \left(  x_2^{2n-3}  (x_2, x_3^{q+m}) 
  +                         (I_{n} : x_3^{q+m} )\right)
\eeq
We now verify that all the terms except  $x_2^{2n-1}  (x_2, x_3^{q+m})$ are in $(I_{n+1} : x_3^{q+m})$.
                     \beqn
                      x_3^{q+m}  \left(   (x_2^{2n-2})  x_3^{2(q+m)}    \right)  
&=&                  (x_2^{2(n-1)} )(x_3^{2(q+m)+1} x_3^{q+m-1} )
\subseteq       f (\p^{n-1} T)
\subseteq       I_{n -1 +2} = I_{n + 1}\\
                      (x_2^2, x_2 x_3^{q+m})       (I_{n} : x_3^{q+m} )  
&\subseteq & (\p T)  (I_{n} : x_3^{q+m} ) 
\subseteq      (I_{n+1} : x_3^{q+m} ) \\
                      x_3^{q+m}  \left( x_3^{q+m+1} x_2^{2n-3}  ( x_2, x_3^{q+m} )   \right) 
  &=    &         x_2^{2(n-2)} \cdot x_2( x_2, x_3^{q+m} ) \cdot x_3^{2(q+m)+1}
  \subseteq        f(\p^{n-1} T)  \subseteq I_{n -1 +2} = I_{n + 1}\\
                                  f (I_{n} : x_3^{q+m} )
&  \subseteq  &   (I_{n+1} : x_3^{q+m} )        
                     \eeqn
\end{proof}

 We  are now ready to describe the generators of $I_{n-1}$ modulo $(I_n : x_3^{q+m})$. 
\label{inductive}
\begin{lemma}
\label{ideal containment}
For all $n \geq 2$,
$${\displaystyle
I_{n-1} = 
 \begin{cases}
 {\displaystyle \sum_{a_2=0}^{ \frac{n-2}{2}}    
     x_3^{(2(q+m)+1) a_2}   x_2^{2(n-1-2a_2) -1}   (x_2, x_3^{q+m})   
+   (I_{n} : x_3^{q+m} )  }
&    \mbox{ if }2 \not |  (n-1 )\\
      \left( x_3^{(2(q+m)+1) \left( \frac{n-1}{2}\right)} \right)
+      {\displaystyle \sum_{a_2=0}^{ \frac{n-3}{2} }    
      x_3^{(2(q+m) + 1)a_2}   x_2^{2(n-1-2a_2) -1}    (x_2, x_3^{q+m})    
 +  (I_{n} : x_3^{q+m} )  } 
 &   \mbox{ if }2 |  (n-1 )
\end{cases}
}.
$$
\end{lemma}
\begin{proof}
From (\ref{definition of Ji}) we get
{\small
\beqn
&&I_{n-1} \\
&=& 
\begin{cases}
{\displaystyle \sum_{a_2=0}^{ \frac{n-2}{2}}  f^{a_2} (\p T)^{n-1-2a_2 }} & 
    \mbox{ if }2 \not |  (n-1 )\\ 
(f ^{\frac{(n-1)}{2}}T)  
+  {\displaystyle \sum_{a_2=0}^{ \frac{n-3}{2} }  f^{a_2} (\p T)^{n-1-2a_2 } } 
&                    \mbox{ if }2 |  (n-1 )
                     \end{cases}\\
&\subseteq &
 \begin{cases}
 {\displaystyle \sum_{a_2=0}^{ \frac{n-2}{2} }    f^{a_2}
      \left( x_2^{2(n-1-2a_2) -1}  (x_2, x_3^{q+m} )  
      + (I_{n-2a_2} : x_3^{q+m} )   \right)    }   & \mbox{ if }2 \not |  (n-1 )\\
  (  x_3^{(2(q+m)+1) )\left( \frac{n-1}{2}\right)} 
+   {\displaystyle \sum_{a_2=0}^{ \frac{n-3}{2}   } (f^{a_2} T)
      \left(   x_2^{2(n-1-2a_2) -1}     (x_2, x_3^{q+m}  )  
 + (I_{n-2a_2} : x_3^{q+m} )  \right)  }
 &   \mbox{ if }2 |  (n-1 )
  \end{cases} \mbox{[by Lemma~\ref{in between step}]}\\
  &\subseteq &
 \begin{cases}
 {\displaystyle \sum_{a_2=0}^{ \frac{n-2}{2}}    
   x_2^{2(n-1-2a_2) -1}   x_3^{(2(q+m)+1) a_2}  (x_2, x_3^{q+m}   )   
+  (I_{n} : x_3^{q+m} )  }& \mbox{ if }2 \not |  (n-1 )\\
   x_3^{(2(q+m) + 1) \left( \frac{n-1}{2}\right)} 
+     {\displaystyle \sum_{a_2=0}^{ \frac{n-3}{2} }    
        x_2^{2(n-1-2a_2) -1}  x_3^{(2(q+m)+ 1)a_2}   (x_2, x_3^{q + m})    
 +  (I_{n} : x_3^{q+m} )  } 
 &   \mbox{ if }2 |  (n-1 ).
  \end{cases}
\eeqn
}
This implies that $I_{n-1} \subseteq RHS$. The other inclusion follows from  Lemma~\ref{ideal containment 1} and checking element-wise. 
\end{proof}
\begin{proposition}
\label{length of last term}
For all $n \geq 1$,
\beqn
\ell \left( \f{I_{n-1} }
                  { (I_n : x_3^{q+m})} \right) = n. 
\eeqn
\end{proposition}
\begin{proof}
From Lemma~\ref{vector basis}, $\ell \left( \f{I_{n-1} }
                  { (I_n : x_3^{q+m})} \right)  =  \dim_{\kk} \left( \f{I_{n-1} } { (I_n : x_3^{q+m})} \right)$,  which is the number of minimal set of generators of 
                  $I_{n-1}/
                  { (I_n : x_3^{q+m})} $.
                 From Lemma~\ref{ideal containment}, we observe that in  the generators of 
$I_{n-1}$ modulo  $(I_n : x_3^{q+m})$, the terms which are of  even degree in $x_2$ are 
  $x_2^{2(n-1-2a_2) }   x_3^{(2(q+m)+1) a_2} $  where $a_2 \leq (n-1)/2$ and they are all distinct. 
  Hence they are all linearly independent.  Similarly, the generators which are of odd degree in 
  $x_2$ are all distinct and form a linearly independent set. 
Hence 
                  \beq
\label{v sp dimension}
\dim_{\kk} \left( \f{I_{n-1} }
                  { (I_n : x_3^{q+m})} \right)
&=& \begin{cases}
2(n/2)  & \mbox{ if } 2 \not  |n-1 \\
1 + [2(n-1)/2]  & \mbox{ if } 2 | n-1
\end{cases}\\
&=& n. 
\eeq
\end{proof}

\begin{proposition}
\label{main theorem}
For all $n \geq 1$, 
\beqn
\ell \left(  \f{T}{I_n} \right)
=     (2(q+m)+1){n+1 \choose 2}.
\eeqn
\end{proposition}
\begin{proof} We prove by induction on $n$. If $n =1$, then 
\beqn
\ell \left(  \f{T}{I_n} \right)
= \ell \left( \f{ \kk[x_2,  x_3]} {(   x_2^2,  x_2 x_3^{m+q}, x_3^{q+m+1} ) } \right)
 = 1 + 2(q+m).
 \eeqn

Now let $n > 1$.
From the  exact sequence
\beqn
        0 
\lrar \f{T}{(I_n: x_3^{q+m})} 
\stackrel{.x_3^{q+m}}{\lrar} \f{T}{I_n}
\lrar \f{T}{I_n + (x_3^{q+m})}
\lrar 0  
\eeqn
we get
\beqn
&&        \ell \left( \f{T}{I_n}\right)\\
&=&      \ell \left( \f{T}{I_n + (x_3^{q+m})}\right)
+            \ell \left( \f{T}{(I_n : x_3^{q+m})}\right)\\
&=&       \ell \left( \f{T}{ (x_3^{q+m} , x_2^{2n}  )}\right)
+            \ell \left( \f{T}{I_{n-1}}\right)
+            \ell \left( \f{I_{n-1}}{(I_n : x_3^{q+m})}\right) 
              \hspace{2.4in} \mbox{[Lemma~\ref{ideal containment 1}]}\\
&=&        2(q+m)n
+             (2(q+m)+1)  {n \choose 2}
+            n 
             \hspace{0.8in} 
             \mbox{[by induction hypothesis and Proposition~\ref{length of last term}]}\\
&= &      (2(q+m)+1) {n + 1 \choose 2}.
\eeqn
\end{proof}

\section{The symbolic powers}
\label{main section}
In this section we explicitly describe the symbolic powers $\p^{(n)}$. 
Using the fact $x_1, x_4$ is a regular sequence in $R$, we  get the results we are interested in  for  the symbolic powers $\p^{(n)}$
(Proposition~\ref{main theorem 1}, Theorem~\ref{symbolic power}). Let
%
\beq
\label{equation of Jn} 
     {\I}_n 
:= \sum_{n_1 + 2 n_2  = n} 
      f^{n_{2}}  \p^{n_1}   \subseteq \p^{(n)}.
\eeq

\begin{proposition}
\label{description of In}
Let $n \geq 1$. Then
\been
\item
\label{description  of In one}
$    {\mathcal I}_n  \subseteq   \p^{(n)}$.

\item
\label{description of In three}
Let $\m = (x_1, x_2, x_3, x_4)$. Then  $({\mathcal I}_n  + (x_1, x_4))$ is an $\m$-primary ideal.
\eeen
\end{proposition}
\begin{proof}
(\ref{description of In one})
As    $(f)    \subseteq \p^{(2)}$  (Lemma~\ref{lemma on f_2}(\ref{lemma on f_2-3})), 
\beq
\label{containment of J}
 \sum_{n_1 + 2 n_2 = n} f^{n_{2} }\p^{n_1}  \subseteq  \sum_{n_1 + 2 n_2 = n} \p^{ 2n_2 + n_1}
= \p^{(n)  }.
\eeq

(\ref{description  of In three}) 
By \eqref{equation of Jn}, 
$\p^{n} \subseteq {\mathcal I}_n$ and  
$(\p^{n} + (x_1, x_4))
= (    ( x_2^{2} ,x_2 x_3^{q+m},  x_3^{q+m+1}  )^{n} ,    x_1, x_4)$ 
which implies that 
$\m  =( \sqrt{\p^{n} + (x_1, x_4)} ) \subseteq (\sqrt{{\mathcal I}_n + (x_1, x_4)}) \subseteq \m $. 
\end{proof}
  
\begin{proposition}
\label{main theorem 1}
For all $n \geq 1$, 
\beqn
         e   \left( (x_1,x_4); \f{R}{\p^{(n)}}\right)
&=&       \ell \left(  \f{R}{\p^{(n)} + (x_1,x_4)}\right)
=  \ell_R \left(\f{R}
                        { ({\mathcal I}_{n}, x_1, x_4) }\right)
= \ell \left( \f{T}{I_n} \right)\\
&=& (2(q+m)+1){n+1 \choose 2}.
\eeqn
\end{proposition}
\begin{proof} 
From  Proposition~\ref{description of In}(\ref{description  of In one}),   ${\mathcal I}_n  \subseteq   \p^{(n)}$.  
Hence,

\beq
\label{equality of all terms 1}\nonumber
 e\left( (x_1, x_4);\f{R}
                         {\p^{(n)}}\right)
&=&    \ell_R \left(\f{R}
                                  {\p^{(n)}+ (x_1, x_4)}\right) 
                                    \hspace{0.9in} [\mbox   {as $R/ \p^{(n)}$  is Cohen-Macaulay}]\\ \nonumber
 &\leq&         \ell_R \left(\f{R}
                                  { ({\mathcal I}_{n}, x_1, x_4) }\right)\\ \nonumber
&=&        \ell_R \left( \f{T}{I_n}\right)\\ \nonumber
                                            &=&        \ell_{R/(x_1, x_4)} \left( \f{T}{I_n}\right)
                                            \hspace{1.7in} \mbox{[as $(x_1, x_4) \subseteq \ann(T/I_n)$]}\\ \nonumber
 &=&              \ell_{T} 
              \left(\f{T}
                       {I_n}\right)      \hspace{3.4in} 
               \mbox{  [\eqref{definition of Ji}]  }  \\ \nonumber    
& = &    (2(q+m)+1) \binom{n+1}{2}   
  \hspace{1.7in} 
               \mbox{[Proposition~\ref{main theorem}]} \\     \nonumber    
 &=& e \left( (x_1, x_4);\f{R}{\p}\right)
              \ell_{R_{\p}}
              \left(\f{R_{\p}}{\p^{n}R_{\p}}\right)   \hspace{1.55in} 
               \mbox{[Lemma~\ref{cohen-macaulay}(\ref{corollary multiplicity})] } \\ \nonumber
&=& e \left( (x_1,x_4);\f{R}
                     {\p}\right)
             \ell_{R_{\p}}
             \left(\f{R_{\p}}
                       {\p^{(n)}R_{\p}}\right)     
                        \hspace{1.1in} [\mbox{since } \p^{(n)}R_{\p}=\p^{n}R_{\p}]   \\          
&=&    e\left( (x_1, x_4);\f{R}
                         {\p^{(n)}}\right)
                \hspace{2.5in} \mbox{[by \cite[1.8]{hio}]}.    
\eeq
Thus equality holds in (\ref{equality of all terms 1}) 
which proves the theorem.
\end{proof}

We end this section by explicitly describing the generators of $\p^{(n)}$ for all $n \geq 1$. 

\begin{theorem}
\label{symbolic power}
For all $n \geq 0$, 
\been
\item
\label{symbolic power one}
 $\p^{(n)} = {\mathcal{I}}_n $.
\item
\label{symbolic power two}
$\p^{(2n)} = (\p^{(2)} )^n$
 and $\p^{(2n+1)} = \p (\p^{(2)})^{n}$, 
  \eeen
\end{theorem}
\begin{proof} 
(\ref{symbolic power one})
By Proposition~\ref{main theorem 1}  we get 
$\p^{(n) }  + (x_1, x_4)= \I_n  + (x_1, x_4). $ Localizing at $\m$ we get 
$(\p^{(n) }  + (x_1, x_4)) R_{\m}=  (\I_n  + (x_1, x_4))R_{\m} $. 
From  Lemma~\ref{cohen-macaulay}(\ref{cohen-macaulay-1}),  we conclude that $x_1R_{\m},  x_4 R_{\m}$ is a regular sequence on $R_{\m}/ \p^{(n)} R_{\m}$. Hence 
\beqn
(\p^{(n) },  x_1) R_{\m} =   (\I_n , x_1)R_{\m} + x_4 ( (\p^{(n) },  x_1)  : x_4 )R_{\m}
=  (\I_n , x_1)R_{\m}  + x_4  (\p^{(n) },  x_1) R_{\m}.
\eeqn
By  Nakayama's Lemma, $(\p^{(n) },  x_1) R_{\m} =  (\I_n , x_1)R_{\m}$. This implies that 
\beqn
\p^{(n) }R_{\m}
= \I_nR_{\m}  + x_1 ( \p^{(n) }  : x_1)R_{\m}
=  \I_n R_{\m} + x_1  \p^{(n)}R_{\m}.
\eeqn 
Once again by Nakayama's lemma, $\p^{(n)} R_{\m}=  {\mathcal I}_nR_{\m}$. This implies that  
$\p^{(n)} R_{\m}/  {\mathcal I}_nR_{\m} = (0) R_{\m}$. As this is a graded module, 
$\p^{(n)}/  {\mathcal I}_n= (0) $ which implies that $\p^{(n)} =  {\mathcal I}_n$.
 
(\ref{symbolic power two}) For all $n \geq 3$, applying Proposition~\ref{description of In}(\ref{description  of In one})   we get
  \beqn
  \p^{(n) } 
  = {\mathcal I}_n 
  = \sum_{a_1 + 2a_2=n} 
         f^{a_2}  \p_1^{a_1}            
 \subseteq      \sum_{a_1 + 2a_2=n} 
         \p^{a_1} (\p^{(2)})^{a_2}
         &\subseteq& \p^{(n)}. 
  \eeqn
   Hence equality holds and 
    ${\displaystyle  \p^{(n)}
= \sum_{a_1 + 2a_2=n} 
  \p^{a_1} (\p^{(2)})^{a_2}  }$. 
  Thus 
   \beqn
 \p^{(2n)}
&=& \sum_{a_1 + 2a_2=2n} 
  \p^{a_1} (\p^{(2)})^{a_2}  
  = \sum_{a_2 = 0}^n \p^{2n-a_2} (\p^{(2)})^{a_2}  \subseteq \sum_{a_2 = 0}^n (\p^{(2)})^{2n} =  \p^{(2n)}\\
   \p^{(2n+1)}
   &=& \sum_{a_1 + 2a_2=2n+1} 
  \p^{a_1} (\p^{(2)})^{a_2}  
  = \sum_{a_2 = 0}^n \p^{2n+1-a_2} (\p^{(2)})^{a_2}  \subseteq  \sum_{a_2 = 0}^n \p  (\p^{(2)})^{2n} = \p \p^{(2n)} \subseteq    \p^{(2n+1)}.
   \eeqn
   
\end{proof}

\begin{corollary}
\label{cor-cohen-macaulay}
For all $n \geq 1$, $R/ \p^{(n)}$ is Cohen-Macaulay
\end{corollary}
\begin{proof} As $x_1, x_4$ is a system of parameters in $R/ \p^{(n)}$ and    $e( (x_1,x_4); \f{R}{\p^{(n)}}) =  \ell \left(  R/ \p^{(n)+ (x_1,x_4)}\right)$ (Theorem~\ref{main theorem 1}),    $R/ \p^{(n)}$ is Cohen-Macaulay. 
\end{proof}

\section{Gorenstein property of symbolic blowup algebras}
In this section we discuss the Gorenstein property of symbolic blowup algebras. If $q=1$, then the curves we are interested in has been studied in \cite{morales-simis}. Our proof here is different. 

Throughout this section $U:= \kk[x_1, x_2, x_3, x_4, u_1, u_2, u_3, v]$ and 
$K := (w_1, w_2, z_1, z_2, z_3)$ where 

\beqn
w_1 &=& x_1 u_1 -  x_2u_2 +  x_3^{q+m} u_3\\
w_2 &=& x_2 u_1 - x_3 u_2 + x_1^{q} x_4^{m} u_3\\
z_1 &=&  x_1 v      -  x_3^{q+m-1} u_1 u_3    + u_2^2, \\
z_2 &=& x_2 v      + x_1^{q -1} x_3^{q+m-1}  x_4^m u_3^2  +  u_1 u_2, \\
z_3  &=&  x_3 v    + x_1^{q-1} x_4^{m} u_2 u_3   +  u_1^2.
\eeqn
Before we prove our main result we prove some preliminary results. 

\begin{lemma}
\label{lemma-cm-rees}
 $U/K$ is Gorenstein.
\end{lemma}
\begin{proof}
Using the Buchsbaum-Eisenbud criterion one can check that minimal free resolution of $U/K$ is 
 \beq
\label{mfr}
0 \lrar U \sta{\phi_3}{\lrar}  U^5\sta{\phi_2} \lrar U^5 \sta{\phi_1}{\lrar} U \lrar \f{U}{K} \lrar 0
\eeq
where
\beqn
\phi_1 = (w_1 ~ w_2 ~ z_1 ~ z_2 ~ z_3), 
\eeqn
\beqn
\phi _2= 
{\left( \begin{array}{ccccc}
0                          & v                            & -x_1^{q-1} x_4^m u_3   & -u_1 &  u_2      \\
-v                          & 0                              &     -u_1                      & u_2  &  -x_3^{q + m-1} u_3 \\  
 x_1^{q-1} x_4^mu_3   & u_1                         &  0                            &   x_3   & - x_2    \\   
  u_1                     &  -u_2                        &  -x_3                       &0        & x_1   \\
  -u_2                     & x_3^{q + m-1} u_3&   x_2                     & - x_1   &  0 \\
\end{array}\right)
}, \hspace{.2in}
\phi_3 = 
\left( {\begin{array}{c}
w_1\\ w_2 \\ z_1 \\z_2 \\ z_3\\
\end{array}}\right).
\eeqn
Moreover, $K$ is generated by the Pfaffians of order $4$ of the anti-symmetic matrix $\phi$ and $R/K$ is Gorenstein. 
\end{proof}

\begin{proposition}
\label{prop-cm-rees}
Let $\tau : \kk[x_1, x_2, x_3, x_4, u_1, u_2, u_3, v]  \to R_{s}(\p)$ be an homomorphism given by $\tau (x_i) = x_i$ ($1 \leq i \leq 4$),  $\tau(u_i) = g_i t$ $1 \leq j \leq 3$  and $\tau(v) = ft^2$.  Then $\ker~\tau = K$.
\end{proposition}
\begin{proof} 
By Lemma~\ref{lemma-cm-rees} all associated primes of $K$ are minimal primes. As $K \subseteq \ker~\tau$ and  $\height(K) =\height(\ker~\tau)=3$, $\ker~\tau$ is a minimal prime of $K$.  

We claim that there exists $a_1, a_2, a_3 \in \kk$, $(a_1, a_2, a_3) \not = (0,0,0)$ such that 
$\alpha = a_1x_1 + a_2x_2 + a_3x_3  \not \in \displaystyle{\bigcup_{P \in \ass(K)}P}$. Otherwise $(x_1, x_2, x_3) \subseteq  \displaystyle{\bigcup_{P \in \ass(K)} P}$ which implies that $(x_1, x_2, x_3)  + K \subseteq \displaystyle{\bigcup_{P \in \ass(K)}P} $ and hence 
$(x_1, x_2, x_3, u_1, u_2) \subseteq \displaystyle{\bigcup_{P \in \ass(K)}P} $. Consequently,  $(x_1, x_2, x_3, u_1, u_2) \subseteq Q$ for some $Q \in \ass(K)$. This implies that $\height(Q) \geq 5$ which leads to a contradiction and proves the claim. 

Fix $\alpha = a_1x_1 + a_2x_2 + a_3x_3  \not \in \displaystyle{\bigcup_{P \in \ass(K)}P}$.
Then
\beqn
&&       \alpha v \\
& = & a_1x_1 v + a_2 x_2 v + a_3x_3 v\\
&= &  a_1z_1 + a_2z_2 + a_3z_3  \\ &&-     
 [a_1 ( -  x_3^{q+m-1}   u_1 u_3    + u_2^2 ) + a_2(  x_1^{q -1} x_3^{q+m-1} x_4^m u_3^2  +  u_1 u_2) 
+ a_3 ( x_1^{q-1} x_4^{m} u_2 u_3   +  u_1^2) ]\\
&=& a_1z_1 + a_2z_2 + a_3z_3  -\beta
\eeqn
where $\beta =a_1 (  -  x_3^{q+m-1}  u_1 u_3    + u_2^2 ) + a_2 (  x_1^{q -1} x_3^{q+m-1} x_4^m u_3^2  +  u_1 u_2)
+ a_3 ( x_1^{q-1} x_4^{m} u_2 u_3   +  u_1^2) $. 
Then
\beqn
              \f{U [1/\alpha]}{ (w_1, w_2, v + \beta/\alpha) }
 \cong    \f{\kk[x_1, x_2, x_3, x_4, u_1, u_2, u_3, v] [1/ \alpha]}{   (w_1, w_2, v + \beta/\alpha) }
\cong   \f{\kk[x_1, x_2, x_3, x_4, u_1, u_2, u_3] [1/ \alpha] }{   (w_1, w_2) }
\cong R(\p)[1 / \alpha]. 
 \eeqn
Recall that  $\p = (g_1, g_2, g_3)$ and $g_1,  g_2, g_3$ form a  $d$-sequence \cite{huneke-d} and 
$$
            R(\p)  = \displaystyle{\bigoplus_{n \geq 0} \p^{n} t^n}
 \cong \f{\kk[x_1, x_2, x_3, x_4, u_1, u_2, u_3] }
 { (w_1, w_2)}. 
$$
Moreover   $R(\p)$ is a domain \cite[Theorem~3.1]{huneke-symm-rees-alg} and   $\dim~ R(\p) = 5$. 
Hence  $(w_1, w_2)$ is a prime ideal of height $2$. Since $\alpha \not \in (w_1, w_2)$, $\height (w_1, w_2) U[1/ \alpha] =2$. 
This implies that  $(w_1, w_2, v + \beta/\alpha)U[1/ \alpha]$ is a prime ideal and     $\height (w_1, w_2, v +  \beta/\alpha)U[1/ \alpha] = 3$. 
In the ring $U[1/ \alpha]$, 
\beqn
(w_1, w_2, v + \beta/\alpha) U[1/ \alpha]
&= &(w_1, w_2, (a_1z_1 + a_2z_2 + a_3z_3) / \alpha) U[ 1 / \alpha]\\
& \subseteq  &K U [1 / \alpha]\\
 &\subseteq & (\ker~\tau )U [ 1/ \alpha]. 
  \eeqn
   Since $\alpha \not \in \ker~\tau$,  $ \height(w_1, w_2, v + \beta/ \alpha) U[1/ \alpha]    =\height ( (\ker~\tau) U[1/ \alpha ]) =3$
and 
\beq
\label{equality of primes}
(w_1, w_2, v + \beta/ \alpha) U[1/ \alpha] =K U [1 /\alpha]    = (\ker~\tau) U[1/ \alpha] . 
\eeq
  By our choice of $\alpha$ and (\ref{equality of primes})
\beqn
 K = K U [1/ \alpha]   \cap U =(\ker~\tau)  U [1/ \alpha]   \cap U =   \ker~\tau. 
\eeqn
\end{proof}

 \begin{theorem} 
 \label{cm-rees}
\been
\item
 \label{cm-rees-zero}
  $\R_s(\p) =R[\p t, ft^2]  $.
\item
 \label{cm-rees-one}
 $\R_s(\p)$ is Cohen-Macaulay.

\item
 \label{cm-rees-two}
$\R_s(\p)$ is Gorenstein.

\item
\label{cm-fiber}
The symbolic fiber cone $F_{s}(\p) = \displaystyle{\bigoplus_{n \geq 0} \p^{(n)}  /  \m \p^{(n)}}$ is Cohen-Macaulay. 
\eeen
\end{theorem}
\begin{proof} 
(\ref{cm-rees-zero}) The proof  follows from  Theorem~\ref{symbolic power}.
 
 (\ref{cm-rees-one})  and (\ref{cm-rees-two}) follows from Lemma~\ref{lemma-cm-rees} and Proposition~\ref{prop-cm-rees}.

(\ref{cm-fiber})
Let $\m = (x_1, x_2, x_3, x_4)$. Then ${\mathcal F}_{s}(\p) \cong U/ (K + \m) \cong \kk[u_1, u_2 , u_3, v]/ (u_1^2, u_1u_2,  u_2^2)$. 
Since $\dim( {\mathcal F}_{s}(\p)  )=2$ and the images of  $u_3$ and $v$ form a regular sequence in ${\mathcal F}_{s}(\p)$, ${\mathcal F}_{s}(\p)$ is Cohen-Macaulay.
\end{proof}

\section{Invariants associated to symbolic powers}
In this section we compute certain invariants namely the resurgence, the Waldschmidt constant, regularity associated to the symbolic powers of $\p$.
Finally we  compare these invariants.

\subsection{Containment}\hfill\\
In order to compare the symbolic powers and ordinary powers  C. Bocci and B. Harbourne  in \cite{BH} defined the resurgence of an ideal $I$ in $R$ as 
 $$
 {\displaystyle 
 \rho(I):=\sup\left\{\frac{n}{r}:  I^{(n)}  \nsubseteq I^r \right\}.}
 $$
 We can also compute the resurgence in the following way.
For any ideal $I \subseteq R$ let  $\rho_n(I):=\min\{r: I^{(n)} \nsubseteq I^r\}.$ Then
 $$
 {\displaystyle 
 \rho(I):=\sup\left\{\frac{n}{\rho_n(I)}:  n \geq 1\right\}.}
 $$

\begin{lemma}
\label{lemma containment-1}
For all $k \geq 1$ and $j=0,1$,
$\p^{( k(2q+2m)   + j)}  \subseteq \p^{k (    2q+2m-1) + j} $ and 
$\p^{( k(2q+2m)   + j)}  \not \subseteq \p^{k (    2q+2m-1) + j +1} $.
 \end{lemma}
\begin{proof} For all $k \geq 1$ and $j=0,1$,  by Lemma~\ref{lemma on f_2}(\ref{lemma on f_2-3})  and Theorem~\ref{symbolic power} we get
\beqn
   \p^{(k(2q+2m))} \p^j
= ((\p^2+ f)^{q+m})^k \p^j
= (\sum_{i=0}^{q+m} f^{q+m-i} \p^{2i})^k \p^j
\subseteq ( \p^{2(q+m-i)-1} \p^{2i})^k \p^j
= \p^{k ( 2q+2m-1) + j} .
\eeqn 
Let $j=0$. We will show that  $\p^{( k(2q+2m) )}  \not \subseteq \p^{k (    2q+2m-1) + 1} $. 
By  Lemma~\ref{lemma on f_2}(\ref{lemma on f_2-2}) and   Theorem~\ref{symbolic power} we get  $f^{k(q+m)} \in \p^{(k(2q+2m))}$. 
 Observe that, 
$f^{k(q+m)} \cong (x_1^{2q+1}  x_4^{2m})^{k(q+m)} \mod(x_3)$. Also 
$$\p^{k (    2q+2m-1) + 1}  =  ( x_1^q x_2  x_4^m,      x_1^{q+1} x_4^{m} ,  x_2^2)^{k (    2q+2m-1) + 1} \mod(x_3)$$
 and hence 
$(x_1^{q+1} x_4^{m})^{  k (    2q+2m-1) + 1}$ is a minimal generator of $\p^{k (    2q+2m-1) + 1} \mod(x_3)$.   
Since 
\beqn
k(2q+1)(q+m) 
&=&  k(2q^2 + 2qm  + q + m)\\
&<&  k(2q^2 + 2qm -q + 2q + 2m -1) + q+1\\
&=& k(q+1)(2q+2m-1)  +q+1,\eeqn
comparing the powers of $x_1$ we get $f^{ k(q+m)} \not \in \p^{k (    2q+2m-1) +1} $. Hence the lemma is true for $j=0$.
Using the similar argument we can show that 
$f^{k(q+m)} g_2 \in  \p^{((k(2q+2m)) +1)} \setminus  \p^{k (    2q+2m-1) +2}$. This shows that the lemma is true for $j=1$.
\end{proof}

\begin{lemma}
\label{lemma containment-2}
For all $k \geq 0$ and  $j =2, \ldots, 2q+2m-1$, 
$\p^{( k(2q+2m)     + j)}  \subseteq \p^{k (    2q+2m-1) + j-1} $
and $\p^{( k(2q+2m)     + j)}   \not \subseteq \p^{k (    2q+2m-1) + j} $.
\end{lemma}
\begin{proof}
Let  $k=0$. If $j = 2j^{\prime}$  and  $j^{\prime} = 1  \ldots q+m-1$, then  by Lemma~\ref{lemma on f_2}(\ref{lemma on f_2-3}) and   Theorem~\ref{symbolic power} we get
\beq
\label{containment k=0 even}
    \p^{(2j^{\prime})} 
= (\p^2 + f)^{ j^{\prime}}
= \sum_{i=0}^{j^{\prime}} f^i \p^{ 2(  j^{\prime} -i)}
\subseteq \p^{2i-1 +  2j^{\prime} -2i}
=  \p^{ 2j^{\prime} -1}
= \p^{j-1}.
 \eeq
If $j = 2j^{\prime} +1$ and     $j^{\prime} = 1, \ldots,  q+m-1$, then from Theorem~\ref{symbolic power}(\ref{symbolic power two}) and (\ref{containment k=0 even}) we get
 \beq
 \label{containment k=0 odd}
\p^{(  2j^{\prime} + 1)}
= \p^{(  2j^{\prime})}\p
 \subseteq \p^{ 2j^{\prime} -1} \p
 = \p^{ 2j^{\prime} } 
 = \p^{j-1}.
 \eeq
Hence the lemma is true for $k=0$.
 
 Now let $k \geq 1$. Then by Theorem~\ref{symbolic power}(\ref{symbolic power two}), induction hypothesis and (\ref{containment k=0 even}) we get
 \beqn
 \p^{  ( k (2q + 2m) +  j)   }  
 &=&  (\p^{  (2q + 2m)  })^k  \p^{(  j) } \subseteq  \p^{ k(2q + 2m -1) } \p^{ j-1}
 =  \p^{ k (2q + 2m -1 ) +  j-1}. 
 \eeqn
 
 We now show that  $\p^{( k(2q+2m)     + j)}   \not \subseteq \p^{k (    2q+2m-1) + j} $.
  Let $j = 2 j^{\prime}$ where $j^{\prime} = 1, \ldots, q+m-1$.  By  
   Lemma~\ref{lemma on f_2}(\ref{lemma on f_2-2}) and Theorem~\ref{symbolic power}(\ref{symbolic power two}) we get  
   \beqn
   f^{ k(q+m)     + j^{\prime}} 
   \in (\p^{(2)})^{( k(q+m)     + j^{\prime})} = \p^{( k(2q+2m)     + 2j^{\prime})}
\eeqn
and
   \beq
   \label{x_1 in f_2}
            f^{k(q+m) +  j^{\prime}} &=& (x_1^{2q+1} x_4^{2m})^{k(q+m) + j^{\prime}} \mod(x_3).
            \eeq
            Since  $\p = (x_1^{q}x_2 x_4^{m}, x_1^{q+1} x_4^m, x_2^2) \mod(x_3)$, 
        \beq
        \label{x_1 in p}
        (x_1^{q+1} x_4^m)^{ k(    2q+2m-1) +  2j^{\prime} } \in \p^{ k(    2q+2m-1) +  2 j^{\prime} }  \mod (x_3)
        \eeq 
         and is a minimal generator. Comparing the power of $x_1$ in (\ref{x_1 in f_2}) and (\ref{x_1 in p}) we get
          $ f^{k(q+m) + j^{\prime}} \not \in \p^{ k(    2q+2m-1) +  2j^{\prime} }  \mod (x_3)$  since
        \beqn
        ( 2q+1) (k(q+m) +  j^{\prime}) 
       -    (q+1) ( k(    2q+2m-1) +  2j^{\prime}) 
       =  -k (q+m) + (q+1)  -1 -j^{\prime}
        < 0
               \eeqn
    Hence $f^{k(q+m) +  j^{\prime}} \not \in \p^{ k(    2q+2m-1) +  2j^{\prime} }$.  

Using the above argument we can show that if $k >1$ and $j=2j^{\prime} + 1$ where $j^{\prime} = 1, \ldots q+m-1$, then  

$ f^{k(q+m) +  j^{\prime}}  g_2 \in \p^{( k(2q+2m)     + j)}    \setminus \p^{k (    2q+2m-1) + j} $. 
 \end{proof}

As a consequence of Lemma~\ref{lemma containment-1} and Lemma~\ref{lemma containment-2}  we verify Conjecture~8.4.3
on \cite{BDHKKASS} for $\p$.  

\begin{lemma}
\label{corollary containment}
Let $r \geq 1$. Then for all $n \geq 2r-1$, $\p^{(n)} \subseteq \p^r$.
\end{lemma}
\begin{proof}
If $n  \geq 2r-1$,   then $\p^{(n)} \subseteq  \p^{(2r-1)}$. Hence it is enough to  show that $\p^{(2r-1)} \subseteq \p^r$. By Theorem \ref{symbolic power} we have $\p^{(2r-1)}=\p(\p^{(2)})^{r-1}\subseteq \p^r$.
  \end{proof}


We now verify that  Conjecture~2.1 in \cite{harb-huneke} holds true in our case.

\begin{corollary}
\label{corollary containment 1}
  $\p^{(3n)} \subseteq \m^{2n} \p^n$. 
\end{corollary}
\begin{proof} Let $r \geq 1$. If $n=2r$, then   by repeatedly applying Theorem~\ref{symbolic power} we get $\p^{  (3(2r))} = (\p^{(2)})^{3r}  = \p^{(2r)} \p^{(4r)} \subseteq \m^{4r} \p^{2r} = \m^{2n} \p^n$ as $\p^{(2)} \subseteq \m^4$ and $\p^{(4r)} = (\p^{(2)})^{2r} \subseteq \p^{2r}$.  

If $n=2r-1$, then 
$
\p^{(3n)} = \p^{(6(r-1))}   \p^{(3)}$. Using the even case  argument, $\p^{(6(r-1))} \subseteq \m^{2 (2r-2)} \p^{2r-2}$. Hence 
$
\p^{(3n)} = \p^{(6(r-1))}   \p^{(3)} \subseteq \m^{2 (2r-2)} \p^{2r-2} \p^{(2)} \p \subseteq \m^{2 (2r-1)} \p^{2r-1}=  \m^{2n} \p^n$.
\end{proof}

We have an improved version  of Corollary~\ref{corollary containment 1}.

\begin{corollary}
\label{corollary containment 2}
For all $n \geq 1$,
  $\p^{(2n)} \subseteq \m^{n} \p^n$ and $\p^{(2n+1)} \subseteq \m^{n} \p^n$ . 
\end{corollary}
\begin{proof} Let $r \geq 1$. Then   by applying Theorem~\ref{symbolic power}(\ref{symbolic power two}) and  (\ref{definition of f2})  we get 
\begin{align*}
                   \p^{  (2n)}
&  =            (\p^{(2)})^{n}  = (\p^2 + (f))^n
 \subseteq  (\m \p)^n = \m^{n} \p^{n} 
 &  \\
                     \p^{(2n + 1) }  
& =               \p \p^{(2n)} 
\subseteq     \m^n \p^{n+1} 
\subseteq \m^n \p^n. & 
\end{align*}
\end{proof}

As a consequence of Lemma~\ref{lemma containment-1} and Lemma~\ref{lemma containment-2} we give the exact value for the resurgence $\rho(\p)$.

\begin{theorem}
\label{theorem-resurgence} For all $q,m \geq 1$, 
 ${\displaystyle \rho(\p)
 = \f{e(R/ \p) -1}{e(R/ \p) -2}.}$
\end{theorem}
\begin{proof}
Let $j=0,1$, $k \geq 1$ and $n_{k,j} = k(2q+2m) +  j$.  Then  by Lemma~\ref{lemma containment-1}, for all $k \geq 1$, $\rho_{k(2q+2m) +  j} (\p) = k(2q+2m-1) +  j +1$. Hence
\beq
\label{eq res 1}
\sup_k\left\{ \f{n_{k,j} } {\rho_{n_{k,j}}(\p)} \right\} =  \sup\left\{\frac{k(2q+2m) +  j}{ k(2q+2m-1) +  j +1}\right\}
=\frac{2q + 2m}{2q + 2m-1}.
\eeq
Let $j=2, \ldots, 2q + 2m-1$,  $k \geq 0$ and $n_{k,j} = k(2q+2m) +  j$. 
Then by Lemma~\ref{lemma containment-2}, for all $k \geq 0$, $\rho_{k(2q+2m) +  j}(\p) = k(2q+2m-1) +  j$. Hence
\beq
\label{eq res 2}
 \sup_k \left\{ \f{n_{k,j} } {\rho_{n_{k,j}}(\p)} \right\}
=  \sup_k\left\{\frac{k(2q+2m) +  j}{ k(2q+2m-1) +  j}\right\}
 =\frac{2q + 2m}{2q + 2m-1} .
\eeq
From (\ref{eq res 1}) and (\ref{eq res 2}) we get
\beqn 
\rho(\p)
=\sup_{k} \left\{ \f{n_{k,j} } {\rho_{n_{k,j}}(\p)}: j=0, \ldots, q+m-1 \right\}
= \frac{2q+2m}{2q+2m-1}.
\eeqn
As $e(R/ \p) = 2(q+m)+1$, the result follows.
\end{proof}

\subsection{Waldschmidt Constant} \hfill\\
For a homogeneous ideal $I\subset R$, let $\alpha(I)$ denote the least generating degree of $I$.  The Waldschmidt constant of $I$ is defined as 
$$
                                           {\gamma}(I)
=\limns  \frac{\alpha(I^{(s)})}{s}.
$$
Here we will compute the Waldschmidt constant for $\p$.

\begin{theorem}
\label{theorem-waldschmidt}
${\gamma}(\p)= \alpha(\p) =2$.
\end{theorem}
\begin{proof}
As $\deg(g_1) = \deg(g_2) \geq \deg(g_3) = 2$, $\alpha(\p)=2$. By Theorem~\ref{symbolic power}(\ref{symbolic power two}), $\p^{(2n)}=(\p^{(2)})^n$ and $\p^{(2n+1)}=\p(\p^{(2)})^n$. Since $\p^{(2)}=(\p^2 +f)$ and $\deg (f)=2(q+m)+1\geq \alpha(\p^2)= 2  \cdot\alpha(\p) = 4$. Thus  ${\displaystyle \f{\alpha(\p^{(2n)}) }{2n}= \f{4n}{2n}=2}$ and 
${\displaystyle \f{\p^{(2n+1)}}{2n+1}=\f{4n+2}{2n+1}=2}$. 
Hence $ \gamma(\p)=2$.
\end{proof}

\subsection{Regularity}\hfill\\
In this subsection we compute the regularity of the symbolic  powers of $\p$. Let $\p T$, $fT$ and $I_n$ as in (\ref{definition of Ji}). 
 We first prove a preliminary result  (Lemma~\ref{reg comparision}) which indicates that it is enough to compute the regularity of 
 $T/I_n$. 
 
 \begin{lemma}
 \label{reg comparision}
 $\reg( R/ \p^{(n)})  = \reg(T/ I_n)$. 
  \end{lemma}
  \begin{proof}
  As $x_1, x_4$ is a regular sequence in $R/\p$, by \mbox{\cite[Remark~4.1]{chardin}}, 
  $$
  \reg \left(  \f{R}{\p^{(n)}} \right)
  =   \reg \left(  \f{R}{\p^{(n) }   + (x_1)}  \right)
  =   \reg \left(  \f{R}{\p^{(n) }   + (x_1, x_4)} \right)
  = \reg \left(  \f{T}{I_{n}  } \right).
  $$
  \end{proof}
Let $G(\F) := \displaystyle{\bigoplus_{n \geq 0} I_n / I_{n+1}}$ be  the associated graded ring  corresponding to the filtration  
${\F}:= \{ I_n\}_{n \geq 0}$.
We show $G(\F)$ is Cohen-Macaulay and this result  is very useful in computing the regularity. 
We first prove a preliminary lemma. 

\begin{lemma}
\label{colon of J_1}
For all $n \geq 1$,
\been
\item
\label{colon of J_1 one}
 $\p^{n} T :(x_2^2)\subseteq \p^{n-1} T$. \newline
 \item
 \label{colon of J_1 two}
 $ ( \p^{n}T : x_3^{2(q+m)+1}) \subseteq \p^{n-2}T$. 
 \eeen
 \end{lemma}
\begin{proof} (\ref{colon of J_1 one}) By Proposition~\ref{prop-herzog},
\beq  
\label{J_1 colon x_2^2}      \nno
                         (\p^{n} T:x_2^2)
&=& \sum_{i=0}^{n-1} \left(x_2^{2(n-i)} x_3^{(q+m)i} (x_2, x_3)^i  : x_2^2\right)
+                                  (x_3^{(q+m)n} (x_2, x_3)^{n} : x_2^2)  \\ \nno
&=& \sum_{i=0}^{n-1} \left(x_2^{2(n-i-1)} x_3^{(q+m)i} (x_2, x_3)^i \right)
+                                  (x_3^{(q+m)n} (x_2, x_3)^{n-2} )   \\
&\subseteq & \p^{n-1}T,\\
\label{colon of j1 with x3} \nno
 &&                  (\p^{n} T : x_3^{2(q+m)+1})  \\ \nno
&=&               \sum_{i=0}^{2} (  x_2^{2(n-i)}  x_3^{(q+m)i} (x_2, x_3)^{i} : x_3^{2(q+m)+1})
+                    \sum_{i=3}^{n} (  x_2^{2(n-i)}  x_3^{(q+m)i} (x_2, x_3)^{i} : x_3^{2(q+m)+1})  
  \\ \nno
&=&                   x_2^{2n}(x_2, x_3) 
+                    \sum_{i=3}^{n+2} (  x_2^{2(n+2-i)}  x_3^{(q+m)(i-3)} (x_2,  x_3)^{i-3})
                        (  x_2^3  x_3^{q+m-1},x_2^2   x_3^{q+m} ,  x_2x_3^{q+m+1},  x_3^{q+m+2}      ) \\ \nno
&\subseteq&   (x_2^{2n})
 +                 \sum_{i=3}^{n} (  x_2^{2(n+2-i)}  x_3^{(q+m)(i-3)} (x_2,  x_2)^{i-3})
                        ( x_2^2,  x_2 x_3^{q+m},  x_3^{q+m+1} )\\
 &\subseteq& \p^{n-2}T.
\eeq
\end{proof}

For any element $r  \in T$, let $r^{\star}$ denote the image in $G(\F)$. 
\begin{theorem}
\label{cohen macaulayness of G}
$G(\F)$ is Cohen-Macaulay.
\end{theorem}
\begin{proof} We first show that $(x_2^2)^{\star}$ is a regular element in $G(\F)$.  
We claim that  $(I_n : x_2^2)  = I_{n-1}$ for all $n \geq 1$. 
Clearly $x_2^2 I_{n-1} \subseteq (\p T) I_{n-1} \subseteq I_n$. For the other inclusion, from Lemma~\ref{colon of J_1}(\ref{colon of J_1 one}) we get 
\beq
\label{colon with x_2^2} \nno                                                        
                       (I_n : x_2^2) 
&=&               \sum_{a_1 + 2a_2 = n} (f^{a_2} T)\left(  \p^{a_1} T  :x_2^2 \right) 
\subseteq
 \sum_{a_1 + 2a_2 = n} (f^{a_2} T)(  \p^{a_1-1} T) 
\subseteq 
I_{n-1}.
\eeq
Let $\overline{\hphantom{xx}}$ denote the image in $T/ (x_2^2)$. Then
\beqn
\f{G(\F) }{{(x_2^2)}^{\star}  }
&\cong& \bigoplus_{n \geq 0} \f{I_n}
                                                 { I_{n+1} +{x_2^2} I_{n-1}} = G( \olin{\F}).
                                                 \eeqn
 To  show that $\olin {x_3^{2(q+m)+1}}$  is a regular element in $G(\olin{\F})$, we  need to verify that  
 \beq
 \label{x3 is regular}    ((I_{n+2}  + {x_2^{2} I_{n}) : (x_3^{2(q+m)+1})) } 
 = I_{n} + x_2^2  I_{n-2}.
 \eeq  
 One can verify that 
$               x_3^{2(q+m)+1}  ( I_{n+2}  + x_2^{2} I_{n}) 
\subseteq fT  (I_{n} + x_2^2  I_{n-2}) 
\subseteq I_{n+2}.$ 
For the other inclusion,  for all $n \geq 0$
\beq
\label{x3 is regular computation}\nno
&&                    { \displaystyle                 ((I_{n+2}  + {x_2^2} I_{n}) : (x_3^{2(q+m)+1})) } \\ \nno
 &=&               \sum_{a_1+ 2a_2 = n+2} ( f^{a_2} \p^{a_1}T  : x_3^{2(q+m)+1})
 +                    \sum_{a_1+ 2a_2 = n}  ( {x_2^{2}  f^{a_2} \p^{a_1}T : (x_3^{2(q+m)+1}))}   \\ \nno
 &=&            \sum_{a_1+ 2a_2 = n+2; a_2 \not =0} (  f^{a_2-1} \p^{a_1}T   )
  +                       (  \p^{n+2} T : x_3^{2(q+m)+1}) \\ \nno
&& +                    \sum_{a_1+ 2a_2 = n; a_2 \not = 0}   (x_2^2  f^{a_2-1} \p^{a_1}T  )
 +                       (x_2^2 \p^{n} T : x_3^{2(q+m)+1})   \\ \nno
 &\subseteq & \sum_{a_1+ 2a_2 = n+2; a_2 \not =0} ( f^{a_2-1}  \p^{a_1}T )
                                                                                   + \p^nT 
 +                   \sum_{a_1+ 2a_2 = n; a_2 \not = 0}  (  x_2^2 f^{a_2-1} \p^{a_1} T )
+                x_2^2  \p^{n-2}  T
    \hspace{1.2in} \mbox{[by \eqref{colon of j1 with x3}]} \\ 
    &\subseteq &  I_{n} + x_2^2  I_{n-2}.
\eeq
\end{proof}

As $x_2^{\star} \in [G(\F)]_1$ is a regular element, we can use it to determine $\reg(T/  (I_{n} + x_2^2) )$.

\begin{lemma}
\label{regularity modulo x_2^2}
Let  $n \geq 1$.  Then 
\beqn
\reg\left( \f{T}{I_{n} + ( x_2^2) } \right) 
= 
\begin{cases}
 2r  \left(  q+m + \f{1}{2} \right)  & \mbox{ if } n = 2r,\\
 (2r - 1)  \left(  q+m + \f{1}{2} \right)  - \f{1}{2}& \mbox{ if }  n = 2r - 1    
\end{cases}.
\eeqn
\end{lemma}
\begin{proof}
By Theorem~\ref{symbolic power},
\beqn
          I_{2r} + (x_2^2)  
&=& (I_2  +   (x_2^2))^r + (x_2^2)
=      (x_2^2, x_3^{2(q+m) + 1})^r + (x_2^2) 
=      ( x_2^2, x_3^{r(2(q+m) + 1)}).
\eeqn

the minimal free resolution of $T/ (I_{2r} +  (x_2^2))$ is 
{\small
\begin{gather}
\xymatrix@C=10pt{
0 \ar[r] 
&  T[-( r(2(q+m) + 1)+2)]   \ar[rrrrr]^(0.5){
\left( 
\begin{array}{c}
x_3^{ r(2(q+m) + 1) }\\
-x_2^2\\
\end{array} 
\right)
  }
   &&&& & {\begin{array}{c}T[-2]   \\\oplus  \\T[-( r(2(q+m) + 1))]  \end{array}} \ar[r] 
  & T\ar[r]
  & {\displaystyle \f{T}{I_{2r} +  (x_2^2)}} \ar[r]
  & 0
  }
\end{gather}}

\normalsize
This implies that 
\beqn
\reg\left( \f{T}{I_{2r} +  (x_2^2)} \right) = r(2(q+m) + 1)= 2r  \left(  q+m + \f{1}{2} \right).
\eeqn 
Let $n= 2r-1$. Then by Theorem~\ref{symbolic power},
\beqn
        I_{2r-1} + ( x_2^2)
&=& (I_1, x_2^2) (I_{2(r-1)}, x_2^2) + (x_2^2)
=     (x_2^2, x_2 x_3^{ (q+m) } , x_3^{ q+m + 1}) ( x_2^2, x_3^{ (r-1)(2(q+m) + 1)}) + (x_2^2)\\
&=& ( x_2^2, x_2 x_3^{ (2r-1)(q+m) + r-1}, x_3^{ (2r-1)(q+m) + r})\\
&=& ( x_2^2, x_2 x_3^{ n(q+m) + r-1}, x_3^{ n(q+m) + r}).
\eeqn
 By Hilbert-Burch Theorem, the minimal free resolution of $T/ I_{n}$ is 
{\small
\begin{gather}
\xymatrix@C=11pt{
0     \ar[r]  
&   T[ - n(q+m)  - r-1 ]^2 \ar[rrrrrrr]^(0.5){  \left( 
\begin{array}{cc}
0 & x_3^{n(q+m)  + (r-1)}\\
-x_3           & -x_2\\
 x_2 & 0
\end{array} \right)}
&&&&&&& { \begin{array}{c} T[-2]  \\ \oplus \\ T[ - n(q+m)  - r ]^2  \end{array}} 
\ar[r] & T
\ar[r] &  {\displaystyle \f{T}{I_{n} + (x_2^2)}}
\ar[r] & 0.
}
\end{gather}}
\normalsize
Hence $\reg(T/ I_{n} + (x_2^2)) = n(q+m+1) + r-1 = n (q+m+1) - \f{1}{2}$.
\end{proof}

\begin{lemma}
\label{lemma mod x_3}
For all $n \geq 1$, $\reg( T/ I_n + (x_3^{2(q+m) + 1} ) )= 2n + 2(q+m)-3$.
\end{lemma}
\begin{proof}
As $(fT) = (x_3^{2(q+m) + 1} ) $, from Theorem~\ref{symbolic power}, for all $n \geq 1$,
\beqn
I_n + (x_3^{2(q+m) + 1} )
&=& (\p T)^n + (x_3^{2(q+m) + 1})
=
 (x_2^{2n}, x_2^{2n-1} x_3^{q+m}, x_2^{2n-2} x_3^{q+m+1}, x_3^{2(q+m) + 1}).  
\eeqn
 By Hilbert-Burch Theorem the minimal free resolution of $ I_n + (x_3^{2(q+m) + 1} )$ is 
{\small
$$
\xymatrix@C=11pt{
0     \ar[r]  
&  {\begin{array}{c}  T[ - (2n + q+m) ]^2 \\     \oplus  \\T[ -( 2n -1+ 2(q+m) )]     \end{array}}
\ar[rrrrr]^(0.5){  \left( 
\begin{array}{cccc}
0 & x_3^{q+m} & 0\\
-x_3               & -x_2 &  0\\
x_2  &      0      & -x_3^{q+m}\\
0 & 0 & x_2^{2n-2}
\end{array} \right)}
&&&&& { \begin{array}{c} T[-(2n)]  \\ \oplus \\ T[ - (2n-1 + q+m) ]^2  \\  \oplus \\ T[- (   2(q+m) + 1)]\end{array}} 
\ar[r] & T
\ar[r] &  {\displaystyle \f{T}{I_{n} + (x_3^{2(q+m) + 1 })}}
\ar[r] & 0.
}
$$}
Hence $\reg( T/ I_n + (x_3^{2(q+m) + 1} ) )= 2n + 2(q+m)-3$. 
\end{proof}

We now use the fact that $({x_3^{2(q+m) + 1}})^{\star} \in [G(\F)]_2 $ is a regular
  element   to compute $\reg(T/ I_n)$. 

\begin{proposition}
\label{prop regularity}
Let  $n \geq 1$. Then
\beqn
  \reg \left( \f{T}{I_{n} }\right) 
= \begin{cases} 2r \left(  q+m  + \f{1}{2} \right) & \mbox{ if } n = 2r,\\
(2r - 1)  \left(  q+m + \f{1}{2} \right)  - \f{1}{2}& \mbox{ if }  n = 2r-1    . 
\end{cases}
\eeqn
\end{proposition}
\begin{proof} Let $n=2r$. We prove the Proposition by induction on $r$. If $r=1$, then the result follows from Lemma~\ref{lemma mod x_3}.
Let $r>1$. 
By Theorem~\ref{cohen macaulayness of G},
  $({x_3^{2(q+m) + 1}})^{\star}$ is a regular
  element   in $G(\F)$.
  Hence,   we have the exact sequence,
\begin{gather}
\begin{gathered}
\label{ses and reg 0}
\xymatrix@C=14pt{
     0 \ar[r]  
     &      {\displaystyle \f{T}{I_{2r-2}}[-2(q+m) -1]} \ar[rrr]^(0.7){.x_3^{2(q+m) + 1}}  
     & & &{ \displaystyle \f{T}{ I_{2r} } }\ar[r]  
     & {\displaystyle \f{T}{ I_{2r} + (x_3^{2(q+m) + 1} ) } }\ar[r]
     .& 0
}.
\end{gathered}
\end{gather}
Then from 
the exact sequence (\ref{ses and reg 0}) 
we get
\beqn
\reg\left(\f{T}{I_{2r}}\right)
&=& \max \left\{    \reg \left(\f{T}{I_{2r-2}} \right) +2(q+m)+1  ,    \reg\left(\f{T}{I_{2r}  + (x_3^{2(q+m) + 1})}\right)     \right\}\\
&=& \max \left\{ (2r-2) \left(q+m+ \f{1}{2}\right) + 2(q+m + \f{1}{2}), 4r + 2(q+m) -3 \right\}
      \hspace{.5in} \mbox{[Lemma~\ref{lemma mod x_3}]}\\
&=& \max \left\{ 2r \left(q+m+ \f{1}{2}\right) , 4r + 2(q+m) -3 \right\}\\
&=& 2r \left(q+m+ \f{1}{2}\right).
\eeqn

Let $n=2r- 1$ and $r \geq 1$. If $r=1$, then the result follows from Corollary~\ref{regularity modulo x_2^2}. Let $r>1$. 
 As $({x_2^2})^{\star}$ is a nonzerodivisor in $G(\F)$,  we have the exact sequence
\beq
\label{ses and reg}
       0
 \lrar \f{T}{I_{2r-2}}[-2]
 \sta{.x_2^2}{\lrar}  \f{T}{I_{2r - 1}}
\lrar  \f{T}{I_{2r - 1} + (x_2^2)}\lrar 0.
\eeq
As all the modules in  \eqref{ses and reg} are Artinian,  
\beqn
\reg(T/I_{2r  - 1}) 
&=& \max \{ \reg(    T/ (I_{2r-2})[-2] ) , \reg( T/ ( I_{2r - 1} + (x_2^2) ))  \}\\
&=& \max \{ (2r-2) \left(q+m + \f{1}{2} \right) +2, (2r - 1) \left(q+m + \f{1}{2}\right)  - \f{1}{2}\}\\
&=& (2r - 1) \left(q+m + \f{1}{2} \right) - \f{1}{2}. 
\eeqn
\end{proof}

\begin{theorem}
\label{theorem-regularity}
Let  $n \geq 1$. Then  ${\displaystyle \reg\left( \f{R}{ \p^{(n) } } \right) = n( e(R/ \p) /2 ) + \theta}$ where 
\beqn
\theta = \begin{cases}
0 & \mbox {if } n \mbox { is even}\\
-1/2  & \mbox {if } n \mbox { is odd}. 
\end{cases}
\eeqn
\end{theorem}
\begin{proof}
From Lemma~\ref{reg comparision} and Proposition~\ref{prop regularity} we get
\beqn
  \reg \left( \f{R}{ \p^{(n) } } \right) 
 = \begin{cases} 2r \left(  q+m  + \f{1}{2} \right) & \mbox{ if } n = 2r,\\
(2r - 1)  \left(  q+m + \f{1}{2} \right)  - \f{1}{2}& \mbox{ if }  n = 2r-1    . 
\end{cases}
\eeqn
Since $e(R/ \p) = 2q  + 1 + 2m$, the result follows.
\end{proof}

As an immediate corollary we have:
\begin{corollary}
\label{limit regularity}
 $\limnn{\displaystyle \f{\reg\left( \f{R}{ \p^{(n) } }\right)}{n} = \f{e(R/ \p) }{2}}$.
\end{corollary}
 
 \subsection{Comparing invariants}\hfill\\
In this subsection  we compare the various invariants. We verify that Theorem~1.2.1(b) does not always hold true if the scheme defined by ideal $I$ is not zero-dimensional.

\begin{lemma}
${\displaystyle \rho(\p) \geq \f{\alpha(\p)}{ {\gamma}(\p) }}$
\end{lemma}
\begin{proof} As ${\alpha(\p)}/{  \gamma(\p)} = 1$ and $\rho(\p) \geq 1$ the result follows. 
\end{proof}

\begin{lemma}
\label{negative result}
If $q=m=1$, then $\rho(\p) \geq \reg(\p)/ {\gamma}(\p) $.
If either $q >1$ or $m>1$, then  $\rho(\p) < \reg(\p)/ {\gamma}(\p) $.
\end{lemma}
\begin{proof} From (\ref{minimal free resolution}) it follows that $\reg(R/ \p) = q+m$. Hence
\beqn
\rho(\p) - \f{\reg(\p)}{ {\gamma}(\p) }
= \f{2q + 2m }{2q+ 2m -1} - \f{q+m}{2}
= \f{(q+m)(5- 2q-2m)}{2(2q+2m-1)}.
\eeqn
If $q=m=1$, then $5-2q-2m=1$. If either $q>1$ or $m>1$, then $5- 2q-2m<0$.
\end{proof}


\begin{thebibliography} {BB}

\bibitem
{BDHKKASS}
T.~Bauer,  S.~ Di Rocco, B.~ Harbourne, M.~ Kapustka, A.~ Knutsen, W.~ Syzdek and T.~Szemberg, 
{\em A primer on Seshadri constants.} Interactions of classical and numerical algebraic geometry, 33-70, 
Contemp. Math., {\bf 496}, Amer. Math. Soc., Providence, RI, 2009. 

\bibitem
{BH0}
C.~Bocci and  B.~Harbourne, {\em  The resurgence of ideals of points and the containment problem.}  Proc. Amer. Math. Soc. {\bf 138} (2010), no. 4, 1175-1190. 

 \bibitem
 {BH} 
C. Bocci and B.~Harbourne, {\em Comparing powers and symbolic powers of ideals}, J. Algebraic Geom. {\bf 19} (2010), no. 3, 399-417.

 \bibitem
 {bocco-waldschmidt}
 C.~Bocci,  S.~Cooper,  E.~Guardo, B.~Harbourne, M.~Janssen,  U.~Nagel, A.~Seceleanu, A. ~Van Tuyl and Thanh Vu, {\em   The Waldschmidt constant for squarefree monomial ideals.} J. Algebraic Combin. {\bf 44} (2016), no. 4, 875-904.

\bibitem 
{chardin}
M.~Chardin, 
{\em Some results and questions on Castelnuovo-Mumford regularity.} Syzygies and Hilbert functions,  1-40, 
Lect. Notes Pure Appl. Math., {\bf 254}, Chapman \& Hall/CRC, Boca Raton, FL, 2007. 

\bibitem
{cutkosky1}
 D.~Cutkosky, J.~ Herzog and N.~V.~ Trung, {\em Asymptotic behaviour of the Castelnuovo-Mumford regularity.}  Compositio Math. {\bf 118} (1999), no. 3, 243-261.
 
\bibitem
{cutkosky2}
D. Cutkosky,
{\em
Irrational asymptotic behaviour of Castelnuovo-Mumford regularity.} 
J. Reine Angew. Math. {\bf 522} (2000), 93-103. 


\bibitem
{duminicki}
M.~Dumnicki, B.~ Harbourne, U.~ Nagel, A.~ Seceleanu, T.~ Szemberg and  H.~Tutaj-Gasi$\acute{n}$ska, {\em  Resurgences for ideals of special point configurations in $\P^{N}$ coming from hyperplane arrangements.} J. Algebra {\bf 443} (2015), 383-394.



 \bibitem
 {ene-herzog} 
 V. Ene and J. Herzog, {\em Gr\"{o}bner bases in commutative algebra.} Graduate Studies in Mathematics, {\bf 130} American Mathematical Society, (2012). 
 
 \bibitem
 {ein-laz-smith}
L.~Ein,  R.~Lazarsfeld, and K.~E.~ Smith,  {\em Uniform bounds and symbolic powers on smooth varieties}. Invent. Math. {\bf 144} (2001), no. 2, 241-252.
 
  \bibitem
  {fatabbi}
 G.~ Fatabbi, Giuliana, B.~Harbourne and 
A.~Lorenzini, {\em Inductively computable unions of fat linear subspaces.} 
J. Pure Appl. Algebra {\bf 219} (2015), no. 12, 5413-5425.

\bibitem
{guardo}
E.~Guardo, B.~Harbourne and A.~ Van Tuyl, {\em Asymptotic resurgences for ideals of positive dimensional subschemes of projective space.} Adv. Math. 246 (2013), 114-127.

\bibitem
{harb-huneke} 
B.~Harbourne and C.~Huneke,  {\em Are symbolic powers highly evolved?}
J. Ramanujan Math. Soc. {\bf 28A} (2013), 247-266. 
 
 \bibitem
 {hochster-huneke}
 M.~Hochster  and  C.~Huneke,
{\em Comparison of symbolic and ordinary powers of ideals}. 
Invent. Math. {\bf 147} (2002), no. 2, 349-369. 

 \bibitem
{hio} 
M.~ Herrmann,  S. Ikeda and U.  Orbanz, {\em  Equimultiplicity and blowing up. An algebraic study. With an appendix by B. Moonen. }(1988 ) Springer-Verlag, Berlin. 
 
\bibitem
  {huneke-symm-rees-alg}
C.~Huneke, {\em  On the Symmetric and Rees Algebra of an Ideal Generated by a $d$-sequence.}
  J. Algebra {\bf 62} (1980), no. 2, 268-275.

\bibitem
{huneke-d}
C.~Huneke,
{\em The theory of d-sequences and powers of ideals.}
Adv. in Math. {\bf 46} (1982), no. 3, 249-279.


\bibitem
{morales}
M.~Morales,
 {\em  Syzygies of monomial curves and a linear diophantine problem of Frobenius}. Max-Planck-Institut f$\ddot{u}$r Mathematik, 
 1987.
 
 \bibitem
 {morales-simis}
 M.~Morales and A.~Simis,
 Arithmetically Cohen-Macaulay monomial curves in $\P^3$. Comm. Algebra {\bf 21} (1993), no. 3, 951-961. 
 
 
 \bibitem
 {schenzel}
 P.~Schenzel, 
 Examples of Gorenstein domains and symbolic powers of monomial space curves. J. Pure Appl. Algebra 71 (1991), no. 2-3, 297-311.
 
 \bibitem
 {waldschmidt}
M.~  Waldschmidt, {\em  Propri$\acute{e}$t$\acute{e}$s arithm$\acute{e}$tiques de fonctions de plusieurs variables. II.} (French) S$\acute{e}$minaire Pierre Lelong (Analyse) ann$\acute{e}$e 1975/76, pp. 108?135. Lecture Notes in Math., Vol. 578, Springer, Berlin, 1977
 
   \end{thebibliography}
\end{document}